\documentclass[11pt,reqno]{amsart}
\usepackage{amscd,amsmath,amsopn,amssymb,amsthm,multicol}
\usepackage{tikz,subdepth,anysize,verbatim,ifthen,xargs,colortbl,float}
\usepackage{longtable,mathtools}
\usepackage[english]{babel}
\usepackage[utf8]{inputenc}
\everymath=\expandafter{\the\everymath\displaystyle}
\usetikzlibrary{decorations.markings,arrows,arrows.meta,bending,calc}
\tikzset{>={Stealth[scale=1.5, bend]}}
\theoremstyle{plain}
\newtheorem{theorem}{Theorem}
\newtheorem{prop}{Proposition}

\newtheorem{cor}{Corollary}
\newtheorem{rk}{Remark}

\renewcommand\a{\alpha}
\renewcommand\b{\beta}
\newcommand\com[1]{}

\newcommand\op[1]{\mathop{\rm #1}\nolimits}
\newcommand\p{\partial}

\newcommand\R{{\mathbb R}}

\arrayrulecolor{gray!50}

\begin{document}

\title{On fractional-linear integrals of geodesics on surfaces}

\author{Boris Kruglikov}
\address{Department of Mathematics and Statistics, UiT the Arctic University of Norway, Troms\o\ 9037, Norway.
\ E-mail: {\tt boris.kruglikov@uit.no}. }

 \begin{abstract}
In this note we give a criterion for the existence of a fractional-linear integral for a geodesic flow on
a Riemannian surface and explain that modulo M\"obius transformations the moduli space of such local integrals
(if nonempty) is either the two-dimensional projective plane or a finite number of points. We will also consider 
explicit examples and discuss a relation of such rational integrals to Killing vectors. 
 \end{abstract}

\maketitle

\section{Introduction}\label{S1}

Integrability of geodesics of a Riemannian metric is a classical problem in mechanics. 
In this paper we will discuss the smallest non-trivial dimension two. 
(For simplicity we will deal with Riemannian metric but all our results extend to the Lorentzian case.)
For integrability of a metric on a surface $M$, it is enough to find only one integral independent of the Hamiltonian of the system.

Globally, there are obstructions for integrability of geodesic flows, namely the genus of $M$ should not exceed one.
Topological classification of integrable Hamiltonian systems with two degrees of freedom is a remarkable
achievement of A.\,T.\ Fomenko and his school, see \cite{BF,BFM} and references therein. This includes 
many classical geodesic flows on two-dimensional sphere and torus.

Locally, any Riemannian metric $g$ is integrable. However the requirment for the integral to be algebraic in momenta imposes
obstructions, see \cite{KM}. For instance, if such an integral $F$ is linear in momenta (Killing vector) then $(M,g)$ 
is a surface of revolution, and if $F$ is quadratic in momenta then $(M,g)$ is Liouville; see \cite{Kr1} for an invariant criterion 
of the existence of such integrals.
Polynomial in momenta integrals of degree $d$ are also called Killing $d$-tensors, they are a classical subject of investigations.

Recently, rational integrals came into a focus of investigation in Hamiltonian mechanics. They appeared in phyiscs
in connection to general relativity \cite{Co} and they were also studied in relation to classical mechanics \cite{Ko}
and the theory of webs \cite{AA}. In the latter work it is shown that the space $\mathcal{F}(g)$ of 
fractional-linear integrals, if non-trivial, 
has dimension either 5 or 3. We observe that the M\"obius group $\op{PSL}(2,\R)$ naturally acts on this space,
the quotient $\mathcal{F}(g)/\op{PSL}(2,\R)$ is well-defined and is either $\R\mathbb{P}^2$ or a finite number of points,
no more than 6. Our proof is based on the notion of relative Killing vector discussed in \cite{Kr2}.

We will also give a criterion for the existence of fractional-linear integrals, similar to an invariant criterion 
for the existence of a Killing 2-tensor from \cite{Kr1}. We apply it to surfaces of revolutions.
Then we address the question whether a given fractional-linear integral is
irreducible. Examples of such integrals are not many, and we will discuss those that appeared in \cite{AS}.
Our observation is that though a given integral may not be a ratio of two Killing vectors, it may be a ratio of 
more complicated Killing tensors. However, generically this does not happen and such a fractional-linear integral
will be genuinely irreducible.

All consideration in this paper are local (the global case is briefly discussed in a special remark), and
in what follows we may assume that $M$ is a two-dimensional disc $D^2$. 
Metric $g$ will be assumed analytic (however for some results smoothness is sufficient).
A part of computations was done using DifferentialGeometry package of Maple; we do not include
long formulae, but they can be accessed through the ancillary files in version 1 of the arXiv posting of this article.

\section{Relative Killing vectors}\label{S2}

With a Riemannian (or pseudo-Riemannian) metric $g=g_{ij}dx^idx^j$ on a surface $M$ 
we associate the Hamiltonian $H(x,p)=\tfrac12g^{ij}(x)p_ip_j$, $(g^{ij})=(g_{ij})^{-1}$,
on $T^*M$ equipped with the canonical symplectic structure $\omega=dx^i\wedge dp_i$. 
This defines the Hamiltonian vector field $X_H=\omega^{-1}dH$, also denoted $\op{sgrad}_\omega H$:
 $$
X_H=H_{p_i}\p_{x^i}-H_{x^i}\p_{p_i}=g^{ij}(x)p_j\p_{x^i}-\p_{x^i}g^{jk}(x)p_jp_k\p_{p_i}
 $$
(summation by repeated indices will be assumed throughout, in our case from 1 to $2=\dim M$).

A function $F=F(x,p)$ is an integral if it is in involution with $H$ wrt the corresponding Poisson structure 
$\{H,F\}=X_H(F)=H_{p_i}F_{x^i}-H_{x^i}F_{p_i}=0$.

Integrals of the form $F(x,p)=f^{i_1\dots i_d}(x)p_{i_1}\cdots p_{i_d}$, polynomial in momenta, 
are known as Killing $d$-tensors. We will also denote $(x^1,x^2)=(x,y)$, $(p_1,p_2)=(p,q)$ in what follows.
In these terms, for $d=1$, the integral $F=u(x,y)p+v(x,y)q$ is also called a Killing vector.

A frational-linear integral is an expression of the form 
 \begin{equation}\label{ratint}
F(x,p)=\frac{P(x,p)}{Q(x,p)}=\frac{u(x,y)p+v(x,y)q}{w(x,y)p+r(x,y)q}, 
 \end{equation}
where $u(x,y),v(x,y),w(x,y),z(x,y)$ are local analytic functions on $M$
(for some results below it suffices to assume those functions smooth) that satisfy $ur-vw\not\equiv0$.

If $F$ is an integral then it is known \cite{AHT,Kr2} that $P$ and $Q$ in \eqref{ratint} are relative Killing vectors
or 1-tensors (also known as generalized Killing tensors), that is, they both satsfy the equation
 \begin{equation}\label{RKV}
\{H,R\}=LR
 \end{equation}
for a linear in momenta function $L(x,p)=\a(x,y)p+\b(x,y)q$ called cofactor (the same for $P$ and $Q$). 
Note that conformal rescaling $R\mapsto e^fR$ results in the change $L\mapsto L+df$ (when we view $L$ as 1-form).
We will identify the pairs $(R,L)$ modulo the following gauge (we will use this equivalence also for cofactor $L$):
 \begin{equation}\label{RLdf}
(R,L)\sim(e^fR,L+df).
 \end{equation}

Let $\mathcal{K}^L_1(g)$ be the space of solutions $R=u(x,y)p+v(x,y)q$ of equation \eqref{RKV} for a fixed $L$.
This equation is of finite type (as will be demonstrated below) and hence its solutions $R$ are given locally by
Cauchy data, that is initial values of $u$ and $v$ at a given point $m\in M$.

 \begin{theorem}\label{Thm1}
The linear space $\mathcal{K}^L_1(g)$ of relative Killing vectors has dimension at most 3;
this upper bound is achieved precisely for metrics $g$ of constant curvature and $L\sim0$.
Otherwise $\dim\mathcal{K}^L_1(g)$ can be 2, 1 or 0.
 \end{theorem}

 \begin{proof}
In isotermal coordinates $g=e^{2\lambda(x,y)}(dx^2+dy^2)$ the Hamiltonian is $H=\tfrac12e^{-2\lambda(x,y)}(p^2+q^2)$. Denoting $(a,b)=e^{2\lambda}(\a,\b)$,
i.e.\ $L(x,p)=e^{-2\lambda}(ap+bq)$, equation \eqref{RKV} has the following form:
 \begin{equation}\label{qqq}
u_x= -(a+\lambda_x)u - \lambda_yv,\ v_x+u_y = -(av+bu),\ v_y = - \lambda_xu - (b+\lambda_y)v.
 \end{equation}
The prolongation of this system is complete, meaning that all second jets $u_{xx},u_{xy},u_{yy}$, $v_{xx},v_{xy},v_{yy}$
are expressed through the lower jets. Thus the system is of finite type.

The first compatibility condition of \eqref{qqq} may be calculated as the multi-bracket \cite{KL2} of these three equations, 
see \cite{Kr1} for a related computation and \cite{KL1} for the general theory of compatibility. The result is 
 \begin{align}\label{gap}
3(u_y-v_x)(b_x-a_y)+
&\bigl((3b+2\lambda_y)(b_x-a_y) + 4(\lambda_{xx}+\lambda_{yy})\lambda_x - 2(\lambda_{xx}+\lambda_{yy})_x + 2(b_x-a_y)_y\bigr)u \\
-& \bigl((3a+2\lambda_x)(b_x-a_y) - 4(\lambda_{xx}+\lambda_{yy})\lambda_y  + 2(\lambda_{xx}+\lambda_{yy})_y +2(b_x-a_y)_x\bigr)v=0.\notag
 \end{align}

If $dL=0$ for 1-form $L=adx+bdy$, i.e.\ $b_x=a_y$, the above equation is of order 0:
 \begin{equation}\label{kuv}
k_xu+ k_yv =0,
 \end{equation}
where $k=\Delta_g\lambda=-e^{-2\lambda}(\lambda_{xx}+\lambda_{yy})$ is the Gaussian curvature 
(half of the scalar curvature); here $\Delta_g$ is the Laplacian of $g$.
Thus we have an alternative: either $k=\op{const}$, $L\sim0$ and the space of solutions $\mathcal{K}^L_1(g)$ 
is 3-dimensional, or $k$ is non-constant and from \eqref{kuv} we conclude that the space of solutions of \eqref{qqq}
is at most 1-dimensional (if non-trivial it is generated by a Killing vector).

On the other hand, if $dL\neq0$, i.e.\ $b_x\neq a_y$, then equation \eqref{gap} is nontrivial, and 
we get a complete system of the first order. Thus the free jets (Cauchy data) are $u,v$ and hence $\dim\mathcal{K}^L_1(g)\leq2$.
As we will see in examples, all dimensions 2, 1 and 0 are realizable for particular $L$.
 \end{proof}

 \begin{rk}
Contrary to the gap phenomenon for (local) Killing vectors, where $\dim\mathcal{K}_1(g)\in\{3,1, 0\}$,
the value 2 is realizable for $\dim\mathcal{K}^L_1(g)$ if $L\not\sim0$. 
 \end{rk}

Consider now the space of relative Killing vectors without specifying cofactor
 $$
\mathcal{K}^{\op{rel}}_1(g) =\cup_L \mathcal{K}^L_1(g).
 $$
 
 \begin{prop}\label{P1}
For any analytic metric $g$ the space $\mathcal{K}^{\op{rel}}_1(g)/\!\sim$ of local moduli of relative Killing vectors 
is non-empty, and it is parametrized by 1 function of 1 argument.
 \end{prop} 
 
 \begin{proof}
We proceed with the notations from the proof of Theorem \ref{Thm1}.
Since the equivalence relation $L\sim0$ is given by $a_y=b_x$, let us impose the constraint $a_x+b_y=0$.
Modulo the gauge it is equivalent to solving the Laplace equation. Solutions of the constraint are parametrized as
$a=f_y$, $b=-f_x$. Substituting this into \eqref{qqq} we obtain a system in the Cauchy-Kovalevskaya form on $(u,v,f)$:
 \begin{align*}
u_y &= \frac{v}{u}u_x +u f_x +v \lambda_x +\frac{v^2}{u}\lambda_y,\\
v_y &= v f_x -u\lambda_x -v\lambda_y,\\
f_y &= -\frac1u u_x -\lambda_x -\frac{v}{u}\lambda_y.
 \end{align*}
Since the coefficients are analytic, there are local solutions, depending on 3 functions of 1 argument.

The remaining gauge freedom is a harmonic function, parametrized by 2 functions of 1 argument,
so the quotient carries the freedom of $3-2=1$ function of 1 argument.
 \end{proof}
 
 \begin{rk}
An alternative proof is as follows. Let $w=v/u$. Use gauge to make $u=1$, $v=w$, leaving $a,b$ un-touched.  
Then equations \eqref{RKV} on $R=p+w(x,y)q$, after elimination $a,b$ imply
 $$
w_y-ww_x+(w^2+1)(\lambda_x+w\lambda_y)=0. 
 $$
This has solutions for any smooth $\lambda$ by the methods of characteristics, and then we get
$a=-\lambda_x-w\lambda_y$, $b=w(\lambda_x+w\lambda_y)-w_x$.
 \end{rk}

Thus to any $g$ there corresponds infinitely many cofactors $L$ with non-trivial relative Killing vectors, but generically
$\dim\mathcal{K}_1^L(g)=1$. To get integrals $F=P/Q\in\mathcal{F}(g)$ we need $L$ with more than one relative Killing vector.
Denote $\mathcal{S}_g=\{L:\dim\mathcal{K}_1^L(g)\geq2\}$ and consider the map
 \begin{equation}\label{ass}
\cup_{L\in \mathcal{S}_g}\op{Gr}_2\mathcal{K}_1^L(g)\ni\langle P,Q\rangle\mapsto [P:Q]=F\in\mathcal{F}(g).
 \end{equation}
This map is $\op{SL}(2,\R)$-equivariant, where the group acts by changing the basis to the left
and by fractional-linear transformations to the right. Actually, the group action to the right is 
 \begin{equation}\label{SL2}
\op{PGL}(2,\R)\ni\begin{pmatrix}\a & \b\\ \gamma & \delta\end{pmatrix}:\frac{P}{Q}\mapsto
\frac{\a P+\b Q}{\gamma P+\delta Q}.
 \end{equation}
 
 \begin{prop}\label{P2}
System \eqref{RKV} on unknowns $(u,v,a,b)$ with the condition $\dim\mathcal{K}_1^L(g)\geq2$ modulo gauge \eqref{RLdf} 
is of finite type. More precisely its solution space is a (finite-dimensional) algebraic manifold.
 \end{prop} 
 
 \begin{proof}
We already discussed the case $dL=0$ in the proof of Theorem \ref{Thm1}, so let us assume $\varrho:=b_x-a_y\neq0$.
Then the first prolongation-projection of \eqref{qqq} is given by
 \begin{gather*}
u_x= -(a+\lambda_x)u - \lambda_yv,\quad\
v_x= \frac{e^{2\lambda}}{3\varrho}\Bigl(k_x u+k_y v\Bigr) 
	+\tfrac13(\lambda+\log\rho)_y u - \bigl(a+\tfrac13(\lambda+\log\rho)_x\bigr)v,\\
u_y= -\frac{e^{2\lambda}}{3\varrho}\Bigl(k_x u+ k_y v\Bigr) 
	- \bigl(b+\tfrac13(\lambda+\log\rho)_y\bigr)u +\tfrac13(\lambda +\log\rho)_x v,\quad
v_y = - \lambda_xu - (b+\lambda_y)v. 
 \end{gather*}

Eliminating the gauge freedom by letting $b=0$, $a_y=w$, we get compatibility of this system as the pair of equations
$\{A_{11}u+A_{12}v=0,A_{21}u+A_{22}v=0\}$. To keep $\dim\mathcal{K}_1^L(g)=2$ these expressions
should vanish identically, and the four equations $A_{ij}=0$ are equivalent to
 \begin{gather*}
w_{xx}= 
\Bigl(3kw -\tfrac73\lambda_xk_y + \lambda_yk_x - k_{xy} + \tfrac53k_y\frac{w_x}{w}\Bigr)e^{2\lambda} 
 + \frac{k_y^2}{3w}e^{4\lambda} + (2\lambda_y^2 -\tfrac23\lambda_x^2 + 2\lambda_{xx})w
 -\tfrac13\lambda_xw_x - \lambda_yw_y +\tfrac43\frac{w_x^2}w,\\
\hskip-15pt w_{xy}=
\Bigl( \frac{4k_yw_y- k_xw_x}{3w} - \tfrac13\lambda_xk_x - \tfrac53\lambda_yk_y - k_{yy}\Bigr)e^{2\lambda}
 - \frac{k_xk_y}{3w}e^{4\lambda} - 3w^2 + (2\lambda_{xy}-\tfrac83\lambda_x\lambda_y)w + \tfrac13(\lambda_xw_y + \lambda_yw_x)
  + \tfrac43\frac{w_xw_y}{w},\\
w_{yy}=
\Bigl(kw -\lambda_xk_y + \tfrac73\lambda_yk_x + k_{xy} - \tfrac53k_x\frac{w_y}{w}\Bigr)e^{2\lambda}
 + \frac{k_x^2}{3w}e^{4\lambda} + (2\lambda_x^2 - \tfrac23\lambda_y^2 - 2\lambda_{xx})w -\lambda_xw_x 
 - \tfrac13\lambda_yw_y + \tfrac43\frac{w_y^2}w,\\
k_x\frac{w_x}w + k_y\frac{w_y}w - 2\lambda_xk_x - 2\lambda_yk_y - k_{xx} - k_{yy}w - 6w^2e^{-2\lambda}=0.
 \end{gather*}

Let $EQ_2$ denote the second order system of the above three equations and $Eq_1$ the last equation of the first order in $w$.
Computing the compatibility conditions for $EQ_2$, i.e.\ equality of mixed derivatives $(w_{xx})_y=(w_{xy})_x$ and
$(w_{xy})_y=(w_{yy})_x$, as well as differentiation of $Eq_1$ by $x,y$ and substitution of $EQ_2$, 
we get four more equations of the first order on $w$ (with coefficients depending differentially on $\lambda$). 

Together with $Eq_1$ we get a system of 5 first order quasilinear PDEs on $w$, of which only three are independent. 
This system is resolved as follows. The first order sub-system, further denoted by $EQ_1$, is
 \begin{align}
w_x &= \frac{180e^{-2\lambda}w^5+h_{13}w^3+h_{12}w^2+h_{11}w+h_{10}}{30k_xw^2+h_{20}},\label{wxeq}\\
w_y &= \frac{(18k_{xy}-18\lambda_xk_y + 42\lambda_yk_x)w^3+h_{32}w^2+h_{31}w+h_{30}}{30k_xw^2+h_{20}},\label{wyeq}
 \end{align}
where $h_{ij}$ are certain differential polynomials in $\lambda$ and $e^\lambda$, but their exact form is not 
important and hence omitted. There is also one zero-order equation $Eq_0$ of the following form:
 \begin{equation}\label{w0eq}
5400w^6+h_{42}w^2+h_{41}w+h_{40}=0.
 \end{equation}
Thus we get that either $w$ is given by the condition of vanishing of both numerator and denominator in
\eqref{wxeq}-\eqref{wyeq} or is given by \eqref{w0eq}. However the former condition defines $w$ up to sign,
and this $w$ should also satisfy \eqref{w0eq}.
This gives only finitely many solutions, in fact no more than 6, hence
finitely many $L\in\mathcal{S}_g$, i.e.\ such that $\dim\mathcal{K}_1^L(g)=2$. 
The solutions are given by algebraic equations, whence the claim.
 \end{proof}

\section{Counting frational-linear integrals}\label{S3}

With each admissible $L\in\mathcal{S}_g$ (that is $\dim\mathcal{K}_1^L(g)\geq2$) 
we associate a family of integrals $F=P/Q$ as in \eqref{ratint} via \eqref{ass}. 
The group $\op{PGL}(2,\R)$ acts on such integrals. Since fractional-linear integrals are
given by the equation on relative Killing vectors, Proposition \ref{P2} implies that the solution space is 
algebraic and the action is algebraic. Hence by Rosenlicht's theorem \cite{R} there exists a rational quotient
$\mathcal{F}(g)/\op{PGL}(2,\R)$, which we think of as the moduli space of fractional-linear integrals.

If $g$ is of constant curvature, then every relative Killing vector is (homologous) to a genuine Killing vector, and 
$F=P/Q$ is a ratio of $P,Q\in\mathcal{K}_1(g):=\mathcal{K}^0_1(g)$. 
Since $\dim\mathcal{K}_1(g)=3$ and the $\op{PGL}(2,\R)$ orbits
correspond to 2-dimensional planes in $\mathcal{K}_1(g)$, the moduli space is 
$\mathcal{F}(g)/\op{PGL}(2,\R)=\op{Gr}_2\mathcal{K}_1(g)=\R\mathbb{P}^2$.

If $g$ is not of constant curvature, then every admissible $L$ determines a fractional-linear integral  
up to the action of $\op{PGL}(2,\R)$, and the number of admissible $L$ is finite by Proposition \ref{P2}.
From its proof we get:

 \begin{cor}\label{Cor1}
The moduli space $\mathcal{F}(g)/\op{PGL}(2,\R)$ is either $\R\mathbb{P}^2$ or a finite set of points,
no more than 6.
 \end{cor}

Thus if the moduli space is non-empty it has dimension 2 or 0. This corresponds to the result of \cite{AA}
according to which the space of fractional-linear integrals has dimension 5 or 3. Our approach, based on
relative Killing vectors, is different and the proof gives a finer description of the space.
 
 \begin{rk}
One can also consider the problem of existence of fractional-linear integrals globally. Recall that 
by Kozlov-Kolokoltsov theorem the geodesic flows on surfaces of genus $>1$ are non-integrable \cite{BF}.
If $M=\mathbb{S}^2$ there are no obstructions for extension of local integrals, and in the case of round metric $g$
the moduli space of fractional-linear integrals is again 2-dimensional $\op{Gr}_2\mathcal{K}_1(g)$.
The same concerns $\R\mathbb{P}^2$. For the torus $\mathbb{T}^2$ the space of Killing vectors is 2-dimensional,
so the moduli space of fractional-linear integrals is a point. For the Klein bottle $\dim\mathcal{K}_1(g)=1$,
therefore the space of fractional-linear integrals is empty.
 \end{rk}

Let us now obtain a criterion of the existence of local fractional-linear integeral.
For this let us return to the end of the proof of Proposition \ref{P2}.

First, compute the compatibility of \eqref{wxeq}-\eqref{wyeq}, expressed as the difference of mixed derivatives
$(w_x)_y-(w_y)_x$ modulo those equations. This gives equation $Eq_0'$ of order zero in $w$.
Next, differentiate equation $Eq_0$ of \eqref{w0eq} by $x$ and $y$ and substitute \eqref{wxeq}-\eqref{wyeq},
obtaining equations $Eq_0''$ and $Eq_0'''$ respectively, both of order zero in $w$. Finally add to those three
the zero order equation $Eq_0$, and denote the obtained system $EQ_0$.

It is polynomial in $w$ of orders, respectively, 10, 8, 7 and 6, with coefficients being differential polynomials
in metric $g$ (in isotermal form, of the conformal factor $\lambda$). These coefficients are not differential invariants,
however the condition that the polynomials have a common root $w$ is invariant. In fact, this condition is given by
the resultant of the polynomials, and it is a differential invariant $\Phi$ of order 6 in the coefficients of metric $g$;
equivalently order 4 in the curvature. This invariant is not scalar, it has three components.

 \begin{theorem}
A fractional-linear integral $F=P/Q$ for the geodesic flow on a surface exists if and only if the above 6th order
differential invariant vanishes $\Phi=0$. 
 \end{theorem}

 \begin{proof}
In the proof of Proposition \ref{P2} we demonstrated that the local existence of a fractional-linear integral
is equivalent to solvability of the complete system $EQ_2+EQ_1+Eq_0$. The compatibility of $EQ_2$ 
was already included in the system. The compatibility of $EQ_1+Eq_0$ is precisely measured by $EQ_0$.
If the system $EQ_0$ has a common zero $w$, it will also satisfy $EQ_2+EQ_1+Eq_0$, since $Eq_0$ is a part of $EQ_0$
and its compatibility with the higher order equations was encounted. Since the system on $w$ is of finite type
this implies the claim.
 \end{proof}
 
 \begin{rk}
We do not show relative invariant $\Phi$ explicitly, as its expression is rather large, even when presented in terms of 
differential invariants of $g$, how this was done for the criterion of Liouville metric in \cite{Kr1}. It is remarkable that the differential
invariant of $g$ governing the existence of fractional-linear integral is almost of the same complexity as that for
the existence of a quadratic integral (for the latter there is one invariant of order 6 and four invariants of order 7). 
For comparison, the criterion for $g$ to admit a Killing vector is given by one invariant of order 4 and one
invariant of order 5, see \cite{Kr1}. See also \cite{Ba} for another criterion of existence rational intergals.
 \end{rk}

Finally let us discuss a relation between relative and absolute Killing vectors for $g$ of nonconstant curvature.
Recall \cite{Koe} (see also \cite{Kr1}) that in the presence of a linear integral the number of  linearly independent 
quadratic ones is even: 2, 4 or 6. In particular, the existence of at least 4 Killing 2-tensors implies the existence of a Killing vector.
As we saw in Proposition \ref{P1} the existence of one relative Killing vector is not a restriction, but the existence of a pair 
of such with the same cofactor $L$ gives a fractional-linear integral.

 \begin{theorem}\label{TKV}
A surface of revolution $(M,g)$ with non-constant curvature does not possess a (nontrivial) fractional-linear integral.  
 \end{theorem}

 \begin{proof}
We can choose the metric in local coordinates $x,y$ such that $\lambda_y=0$ for the conformal factor of $g$, so 
 $$
H=\frac12e^{-2\lambda(x)}(p^2+q^2). 
 $$
Then we can proceed along the lines of proof of Proposition\ref{P2}, but the computations are easier as 
all derivatives of the curvature invariants along $y$ vanish. 

Moreover, the differential invariants of $g$ in those coordinates become of lower order.
Indeed, the curvature $k=-e^{-2\lambda}\lambda_{xx}$ and its derivatives along the vector field
$\op{grad}k=-2e^{-4\lambda}(\lambda_{xxx}-2\lambda_x\lambda_{xx})\p_x$ are invariants, but
the expression $j=e^{-\lambda}\lambda_x$ is also an invariant. Up to the sign it is equal to the length of
the commutator of $V$ and $JV$, where $V$ is a unit vector along foliation $\{k=\op{const}\}$ and $J$
is the operator of rotation by $\pi/2$ (given the orientation).

Expressing the results of computations in the proof of Proposition \ref{P2} 
via differential invariants makes the formulae easier. For instance, the condition of
vanishing of both denominators of \eqref{wyeq} is 
 $$
w=\pm\frac1{2\sqrt{30}}e^{2\lambda}\sqrt{k_x^{-1}k_{xx}-8jk_x^{1/2}}.
 $$ 
This implies $w_y=0$ and substituting this into the numerator of \eqref{wyeq} yields $k_x=0$, i.e.\ constant curvature.

Similarly, the vanishing of $Eq_0$ yields $w=\pm\frac1{2\sqrt{15}}e^{2\lambda}\sqrt{8jk_x^{1/2}-k_x^{-1}k_{xx}}$
and substitution into other equations of $EQ_0$ yeilds expressions that cannot be zero simulteneously.
In other words, $\Phi\neq0$ unless $k=\op{const}$ on $M$.
Hence no nontrivial relative Killing vectors $R$ exist, unless $L\sim0$, in which case we get one Killing vector.
 \end{proof}

\section{Examples}\label{S4}

Here we demonstrate how to use the criterion for the existence of fractional-linear integral via vanishing of the invariant $\Phi$.

\smallskip

{\bf 1.} Consider at first the metric $g=e^{2x}(J_0(y)^2+J_1^2(y))\cdot(dx^2+dy^2)$ derived in \cite{AS},
where $J_i(y)$ are Bessel functions of the first kind, i.e.\ solutions of the differential equations
 $$
J_0'(y)=-J_1(y),\quad J_1'(y)=J_0(y)-\frac1yJ_1(y),  
 $$
with initial conditions $J_0(0)=1$, $J_1(0)=0$, $J_1'(0)=\frac12$. The corresponding Hamiltonian is 
 $$
H=\frac{e^{-2x}}2\frac{p^2+q^2}{J_0(y)^2+J_1(y)^2}. 
 $$
There are no Killing vectors. One can verify this directly by computing the Jacobian of the curvature $k$
and the square norm of its gradient $|\nabla k|^2_g$, which is a non-zero expression of fourth order.

The following fractional-linear integral was found in \cite{AS}:
 \begin{equation}\label{ex1}
F=\frac{P}{Q},\quad P = (xp+yq)J_1(y) - (yp-xq)J_0(y),\ Q = J_1(y)p + J_0(y)q.
 \end{equation}

Thus $P$ and $Q$ are relative Killing vectors, they satisfy the equation on $R$:
 \begin{equation}\label{ex1a}
\{H,R\}= \frac{e^{-2x}}{y(J_0^2(y) + J_1^2(y))^2}\Bigl(
\bigl(-y(J_0^2(y)+J_1^2(y))+J_0(y)J_1(y)\bigr)p + J_1^2(y)q \Bigr)R
 \end{equation} 
 
We will look for general relative Killing vector in the form (with gauge partially eliminated)
 \begin{equation}\label{ex1R}
R=u(x,y)p+v(x,y)q,\quad L=\frac{e^{-2x}}{J_0^2(y) + J_1^2(y)}a(x,y)p.
 \end{equation}
System \eqref{qqq} becomes
 $$
u_x = -(a+1)u +\frac{J_1^2(y)}{y(J_0^2(y)+J_1^2(y)}v,\quad
v_x+u_y = -av,\quad 
v_y = -u +\frac{J_1^2(y)}{y(J_0^2(y)+J_1^2(y)}v 
 $$
We omit expressions for $EQ_2$ and $EQ_1$ from the general theory above, also passing to $w=a_y$. 
The first expression of the first order $EQ_0$ factorizes as $\Psi_1\cdot\Psi_2$, where (we omit inessential terms $h_i$)
 \begin{align*}
\Psi_1 &= y^2\bigl(J_0^2(y)+J_1^2(y)\bigr)^2w+y\bigl(J_1^4(y)-J_0^4(y)\bigr)+2J_0^3(y)J_1(y),\\
\Psi_2 &= 25y^{10}\bigl(J_0^2(y)+J_1^2(y)\bigr)^{10}w^5+ 
h_4w^4+h_3w^3+h_2w^2+h_1w+h_0.
 \end{align*}
Prolonging $\Psi_2$, i.e.\ taking $D_x\Psi_2$, $D_y\Psi_2$ and evaluating them on $EQ_1$ we get 
three polynomials in $w$ with coefficients being functions of $x$ and $y$ (through exponential in $x$ and Bessel in $y$). 
Each of the two new polynomials contain $\Psi_1$ as a factor, but the remaining parts are incompatible: the three
polynomials have no common root $w$ (direct verification by computing the resolvent).

From $\Psi_1=0$ we find
 $$
a_y=w= \frac{y\bigl(J_0^4(y)-J_1^4(y)\bigr)-2J_0^3(y)J_1(y)}{y^2\bigl(J_0^2(y)+J_1^2(y)\bigr)^2}
 $$
and by integrating 
 $$
a=\frac{J_0(y)J_1(y)}{y\bigl(J_0^2(y)+J_1^2(y)\bigr)}.
 $$
Solving equation \eqref{ex1R} we get two independent solutions 
 $$
\frac{(xp+yq)J_1(y) - (yp-xq)J_0(y)}{e^x\sqrt{J_0^2(y)+J_1^2(y)}},\quad
\frac{J_1(y)p + J_0(y)q}{e^x\sqrt{J_0^2(y)+J_1^2(y)}}
 $$
and clearing denominators by gauge we recover the initial relative Killing vectors.
Consequently, up to M\"obius transformations, the fractional rational integral is unique.
 
\medskip 

{\bf 2.} Consider the Hamiltonias (they also correspond to geodesic flows) 
 $$
H_1=\frac{p^2+q^2}{x^2+y^2+b},\quad
H_2=\frac{p^2+q^2}{x^4+y^4+b},
 $$
where $b$ is a real parameter (may be zero).
 
The first Hamiltonian admits 1 linear integral and 4 quadratic integrals, but it does not admit a fractional-linear integral.
This can be obtained through the mentioned criterion, but also follows from Theorem \ref{TKV}.

The second Hamiltonian does not admit Killing vectors and it has only two Killing 2-tensors counting the Hamiltonian
(as it corresponds to a Liouville metric).
Here we compute the existence of fractional-linear integral by the method of the previous section. 
Going through the computation we obtain that the invariant $\Phi$ is nonzero, hence no fractional-linear integrals.

Let us note that the following examples
 $$
H_1^\epsilon=\frac{p^2+q^2}{x^2+\epsilon y^2},\quad
H_2^\epsilon=\frac{p^2+q^2}{x^4+\epsilon y^4},
 $$
are deformations of $H_1,H_2$ (with $b=0$), and the same computation shows that for small parameter $\epsilon$
(and also a perturbation of those by $b$) the systems still do not admit fractional-linear integrals.
However in \cite{AS} a fractional-linear integral for each Hamiltonian was obtained when $\epsilon=4$.
Irreducibility of those through Killing tensors is further discussed in \cite{Kr2}.

\section{Outlook}\label{S5}

Rational integrals were already discussed by Darboux \cite{Da}. 
Their relation with relative (or generalized) Killing tensors was discussed in
\cite{AHT} and further in \cite{Kr2}. For systems polynomial in all variables the latter also appeared in \cite{MP}
under the name of Darboux polynomials.

In this note we discussed a particular kind of those: fractional-linear integrals
and corresponding relative Killing vectors. We obtained certain results on the behavior of moduli spaces of
such integrals on surfaces. Let us mention some important open problems. 

\smallskip

Q1: The bound 6 in Corollary \ref{Cor1} of Proposition \ref{P2} seems not to be sharp, at least no example realizing it is known.
What is the actual bound? Is it true that for $(M,g)$ of nonconstant curvature the moduli space is connected, i.e.\
all rational integrals are equivalent with respect to the M\"obius group $\op{PGL}(2,\R)$? 
(Equivalently, are all cofactors $L$ cohomologous and the frational-linear integral essentially unique?)

\smallskip

Q2: Example 1 of Section \ref{S4} is irreducible in the sense that the fractional-linear integral considered there
is not a ratio of (absolute) Killing vectors. But one cannot exclude that it is not a ratio of higher degree Killing tensors.
For the above Hamiltonins $H_1$ with $b=0$ such an effect was shown in \cite{Kr2}. 
Can one demonstrate such general irreducibility for this or others examples from \cite{AS}? 
While one expects the positive answer, a functional dimension count as in \cite{Ko} cannot guarantee this.
However topological arguments, namely non-resonanse behavior of the trajectories, will be sufficient.

\smallskip

Q3: Integrable geodesic flows can exist only on surfaces of small genus. For orientable $M$ these are $\mathbb{S}^2$
and $\mathbb{T}^2$. However metrics with higher degree Killing tensors are known only on the sphere, and
it is conjectured that they cannot exist on the torus, see \cite{BF} for a discussion. What about rational integrals: can
they be realized for a nonconstant curvature metric on these surfaces?

\smallskip

These questions are also interesting in higher dimensions, but are more challenging there.

 \bigskip

{\bf Acknowledgment.}
I thank Vladimir Matveev for helpful comments.
The work was initiated when the author visited Brasil by the invitation of Sergey Agafonov,
supported by his grant FAPESP 2022/12813-5. I am grateful for hospitality of IMPA at Rio de Janeiro
and UNESP at campus S\~ao Jos\'e do Rio Preto.

The research leading to our results was partially supported by the Tromsø Research Foundation 
(project “Pure Mathematics in Norway”) and the UiT Aurora project MASCOT.


\begin{thebibliography}{50}

\bibitem{AA}
S.\ Agafonov, T.\ Alves, {\it Fractional-linear integrals of geodesic flows on surfaces and Nakai's geodesic 4-webs},
Advances in Geometry, https://doi.org/10.1515/advgeom-2024-0008 (2024).

\bibitem{AS}
S.\ Agapov, V.\ Shubin, {\it Rational integrals of 2-dimensional geodesic flows: new examples}, 
J.\ Geom.\ Phys.\ {\bf 170}, 104389 (2021).

\bibitem{AHT}
A.\ Aoki, T.\ Houri, K.\ Tomoda, {\it Rational first integrals of geodesic equations and generalised hidden symmetries},
Class.\ Quantum Grav.\ {\bf 33}, 195003 (2016).

\bibitem{Ba}
Yu.\ Bagderina, {\it Rational integrals of the second degree of two-dimensional geodesic equations}, 
Sib.\ Elektron.\ Mat.\ Izv.\ {\bf 14}, 33--40 (2017).

\bibitem{BF}
A.\,V.\ Bolsinov, A.\,T.\ Fomenko, {\it Integrable geodesic flows on two-dimensional surfaces}, 
Editorial URSS (1999); Monogr.\ Contemp.\ Math., New York: Consultants Bureau (2000).

\bibitem{BFM}
A.\,V.\ Bolsinov, A.\,T.\ Fomenko, V.\ Matveev, {\it Two-dimensional Riemannian metrics with integrable geodesic flows. 
Local and global geometry}, Mat.\ Sb.\ {\bf 189}, no.\ 10, 5--32 (1998).

\bibitem{Co}
C.\,D.\ Collinson, {\it A note on the integrability conditions for the existence of rational first integrals
of the geodesic equations in a Riemannian space}, Gen.\ Rel.\ Grav.\ {\bf 18}, 207 (1986).

\bibitem{Da}
G.\ Darboux, {\it Lecons sur la theorie generale des surfaces III}, Chelsea Publishing (1896).

\bibitem{Koe}
G.\ Koenigs, {\it Sur les géodesiques a intégrales quadratiques}, Bull.\ Soc.\ Philom.\ Paris (8) {\bf 5}, 26–28 (1893);
also: Note II in: G. Darboux, {\it Lecons sur la theorie generale des surfaces IV}, Chelsea Publishing (1896).

\bibitem{Ko}
V.\ Kozlov, {\it On rational integrals of geodesic flows}, Regul.\ Chaotic Dyn.\ {\bf 19}, 601--606 (2014). 

\bibitem{Kr1}
B.\ Kruglikov, {\it Invariant characterization of Liouville metrics and polynomial integrals}, 
J.\ Geom.\ Phys.\ {\bf 58}, 979--995  (2008). 

\bibitem{Kr2}
B.\ Kruglikov, {\it Rational first integrals and relative Killing tensors}, 
arXiv:2412.04151 (2024). 

\bibitem{KL1}
B.\ Kruglikov, V.\ Lychagin, {\it Geometry of Differential equations},
Handbook on Global Analysis, D.Krupka and D.Saunders Eds., {\bf 1214}, 725-771, Elsevier Sci. (2008)

\bibitem{KL2}
B.\ Kruglikov, V.\ Lychagin, {\it Compatibility, multi-brackets and integrability of systems of PDEs},
Acta Appl.\ Math.\ {\bf 109}, 151--196 (2010).

\bibitem{KM}
B.\ Kruglikov, V.\ Matveev, {\it The geodesic flow of a generic metric does not admit nontrivial integrals polynomial
in momenta}, Nonlinearity {\bf 29}, 1755--1768 (2016).

\bibitem{MP}
A.\,J.\ Maciejewski, M.\ Przybylska, {\it Darboux polynomials and first integrals of natural polynomial
Hamiltonian systems}, Phys.\ Lett.\ A {\bf 326}, 219 (2004).

\bibitem{R}
M.\ Rosenlicht, {\it Some basic theorems on algebraic groups}, Amer.\ J.\ Math.\ {\bf 78}, 401--443 (1956).

 \end{thebibliography}
\end{document}